\documentclass[10pt,twoside]{amsart} 
\usepackage{amsmath, amsthm, amscd, amsfonts, amssymb, graphicx, color}
\usepackage[bookmarksnumbered, colorlinks, plainpages]{hyperref}

\newtheorem{theorem}{Theorem}[section]
\newtheorem{lemma}[theorem]{Lemma}
\newtheorem{proposition}[theorem]{Proposition}
\newtheorem{corollary}[theorem]{Corollary}

\theoremstyle{definition}
\newtheorem{definition}[theorem]{Definition}
\newtheorem{example}[theorem]{Example}

\newtheorem{remark}[theorem]{Remark}

\numberwithin{equation}{section}

\def\<{\langle}
\def\>{\rangle}

\long\def\alert#1{\smallskip{\hskip\parindent\vrule%
\vbox{\advance\hsize-2\parindent\hrule\smallskip\parindent.4\parindent%
\narrower\noindent#1\smallskip\hrule}\vrule\hfill}\smallskip}

\begin{document} 
\title[Uniformly 2-absorbing Primary Ideals of Commutative Rings]{Uniformly
2-absorbing Primary Ideals of Commutative Rings}
\author[Mostafanasab, Tekir and ULUCAK]{Hojjat Mostafanasab, \"{U}nsal Tekir and G\"{U}L\c{S}EN ULUCAK}


\begin{abstract}
In this study, we introduce the concept of ``uniformly 2-absorbing primary
ideals'' of commutative rings, which imposes a certain boundedness condition
on the usual notion of 2-absorbing primary ideals of commutative rings. Then
we investigate some properties of uniformly 2-absorbing primary ideals of
commutative rings with examples. Also, we investigate a specific kind of uniformly 
2-absorbing primary ideals by the name of ``special 2-absorbing primary ideals''.
\end{abstract}

\maketitle

\section{\protect\bigskip Introduction}

Throughout this paper, we assume that all rings are commutative with $1\neq
0 $. Let $R$ be a commutative ring. An ideal $I$ of $R$ is  a  {\it proper ideal} if 
$I\neq R$. Then $Z_{I}(R)=\{r\in R~|~rs\in I$ for some $s\in R\backslash I\}$
for a proper ideal $I$ of $R$. Additively, if $I$ is an
ideal of commutative ring $R$, then {\it the radical of} $I$ is given by $\sqrt{I}%
=\{r\in R~|~r^{n}\in I$ for some positive integer $n\}$. 
Let $I,J$ be two ideals of $R$. We will
denote by $(I:_{R}J)$, the set
of all $r\in R$ such that $rJ\subseteq I$.

Cox and Hetzel have introduced uniformly primary ideals of a
commutative ring with nonzero identity in \textrm{\cite{CH}}.
They said that a proper ideal $Q$ of the commutative ring $R$ is {\it uniformly primary} if there exists a positive integer $n$
such that whenever $r,s\in R$ satisfy $rs\in Q$ and $r\notin Q$, then $s^n\in Q$.
A uniformly primary ideal $Q$  has order $N$ and write $ord_R(Q)=N$, 
or simply $ord( Q) = N$ if the ring $R$ is understood, if $N$ is the smallest positive
integer for which the aforementioned property holds.

Badawi \cite{B} said that a proper ideal $I$ of $R$ is a {\it 2-absorbing ideal of}  
$R$ if whenever $a, b, c\in R$ and $abc\in I$, then $ab\in I$ or
$ac\in I$ or $bc\in I$. He proved that $I$ is a $2$-absorbing
ideal of $R$ if and only if whenever $I_1, I_2, I_3$ are ideals of
$R$ with $I_1I_2I_3\subseteq I$, then $I_1I_2\subseteq I$ or
$I_1I_3\subseteq I$ or $I_2I_3\subseteq I$. Anderson and Badawi
\cite{AB1} generalized the notion of $2$-absorbing ideals to
$n$-absorbing ideals. A proper ideal $I$ of $R$ is called an $n$-{\it absorbing} (resp.
{\it a strongly} $n$-{\it absorbing}) {\it ideal} if whenever $x_1\cdots x_{n+1}\in I$ for
$x_1,\dots,x_{n+1}\in R$ (resp. $I_{1}\cdots I_{n+1}\subseteq I$ for
ideals $I_1,\dots,I_{n+1}$ of $R$), then there are $n$ of the
$x_i$'s (resp. $n$ of the $I_i$'s) whose product is in $I$.
Badawi et. al. \cite{BTY} defined a proper ideal $I$ of $R$ to be a {\it 2-absorbing primary 
ideal} of $R$ if whenever $a, b, c\in R$ and $abc\in I$, then either $ab\in I$ or $%
ac\in \sqrt{I}$ or $bc\in \sqrt{I}$. Let $I$ be a 2-absorbing primary ideal of $R$. Then $P=\sqrt{I}$
is a 2-absorbing ideal of $R$ by \textrm{\cite[Theorem 2.2]{BTY}}: We say that $I$ is a $P$-2-absorbing primary ideal
of $R$. For more studies concerning 2-absorbing (submodules) ideals we refer to \cite{YB},\cite{Mos},\cite{nas},\cite{YF},\cite{YF2}.
These concepts motivate us to introduce a generalization of
uniformly primary ideals. 
A proper ideal $Q$ of $R$
is said to be a {\it uniformly 2-absorbing primary ideal of} $R$ if there exists a
positive integer $n$ such that whenever $a,b,c\in R$ satisfy $abc\in Q$, $%
ab\notin Q$ and $ac\notin \sqrt{Q}$, then $(bc)^{n}\in Q$.
In particular, if for $n=1$ the above property holds, then we say that 
$Q$ is a {\it special 2-absorbing primary ideal of} $R$.

In section 2, 
we introduce the concepts of uniformly 2-absorbing primary ideals and Noether
strongly 2-absorbing primary ideals. Then we
investigate the relationship between uniformly 2-absorbing primary ideals,
Noether strongly 2-absorbing primary ideals and 2-absorbing primary 
ideals. After that, in Theorem \ref{main1} we characterize uniformly 2-absorbing primary ideals.
We show that if $Q_1,~Q_2$ are uniformly primary ideals
of ring $R$, then $Q_1\cap Q_2$ and $Q_1Q_2$ are uniformly 2-absorbing primary ideals of $R$, 
Theorem \ref{intersection}.  Let $R=R_1\times R_2$, where $R_1$ and $R_2$
are rings with $1\neq0$. It is shown (Theorem \ref{product1}) that a proper ideal $Q$ of $R$
is a uniformly 2-absorbing primary ideal of $R$ if and only if either $Q=Q_1\times R_2$ 
for some uniformly 2-absorbing primary ideal $Q_1$ of $R_1$ or $Q=R_1\times Q_2$ for some uniformly 2-absorbing primary
ideal $Q_2$ of $R_2$ or $Q=Q_1\times Q_2$ for some uniformly primary ideal $
Q_1$ of $R_1$ and some uniformly primary ideal $Q_2$ of $R_2$.

In section 3, we give some properties of special 2-absorbing primary ideals. For example, 
in Theorem \ref{special} we show that $Q$ is a special 2-absorbing primary ideal of $R$ 
if and only if for every ideals $I,J,K$ of $R$, $IJK\subseteq Q$ implies that either $
IJ\subseteq \sqrt{Q}$ or $IK\subseteq Q$ or $JK\subseteq Q$. We prove that
if $Q$ is a special 2-absorbing primary ideal of $R$ and $
x\in R\backslash \sqrt{Q}$, then $(Q:_Rx)$ is a special 2-absorbing primary ideal of $R$, Theorem \ref{main2}.
It is proved (Theorem \ref{main3}) that an irreducible ideal $Q$ of $R$ is special
2-absorbing primary if and only if $(Q:_Rx)=(Q:_Rx^2)$ for every $x\in
R\backslash\sqrt{Q}$. Let $R$ be a Pr\"{u}fer domain and $I$ be an ideal of $R$. 
In Corollary \ref{last} we show that $Q$ is a special 2-absorbing primary ideal of $R$
if and only if $Q[X]$ is a special 2-absorbing primary ideal of $R[X]$.
\bigskip

\section{Uniformly 2-absorbing primary ideals}

Let $Q$ be a $P$-primary ideal of $R$. We recall from \cite{CH} that $Q$ is a {\it Noether strongly 
primary ideal of} $R$ if $P^{n}\subseteq Q$ for some positive integer $n$.
We say that $N$ is the exponent of $Q$ if $N$ is the
smallest positive integer for which the above property holds and it is denoted
by $\mathfrak{e}(Q)=N$.
\begin{definition}
Let $Q$ be a proper ideal of a commutative ring $R$.

\begin{enumerate}
\item $Q$ is a {\it uniformly 2-absorbing primary ideal of} $R$ if there exists a
positive integer $n$ such that whenever $a,b,c\in R$ satisfy $abc\in Q$, $%
ab\notin Q$ and $ac\notin\sqrt{Q}$, then $(bc)^{n}\in Q$. We call that $N$
is order of $Q$ if $N$ is the smallest positive integer for which the above
property holds and it is denoted by $\mbox{2-}ord_{R}(Q)=N$ or $\mbox{2-}%
ord(Q)=N$.

\item $P$-2-absorbing primary ideal $Q$ is a {\it Noether strongly 2-absorbing primary ideal of} $R$ if $P^{n}\subseteq Q$ for some positive integer $n$.
We say that $N$ is the exponent of $Q$ if $N$ is the
smallest positive integer for which the above property holds and it is denoted
by 2-$\mathfrak{e}(Q)=N$.
\end{enumerate}
\end{definition}

A {\it valuation ring} is an integral domain $V$ such that for every element $x$ of its field of fractions $K$, at least one of $x$ or $x^{-1}$ belongs to $K$.
\begin{proposition}
Let $V$ be a valuation ring with the quotient field $K$ and let $Q$ be a
proper ideal of $V$.

\begin{enumerate}
\item Q is a uniformly 2-absorbing primary ideal of $V$;

\item There exists a positive integer $n$ such that for every $x,y,z\in K$
whenever $xyz\in Q$ and $xy\notin Q$, then $xz\in\sqrt{Q}$ or $(yz)^n\in Q$.
\end{enumerate}
\end{proposition}

\begin{proof}
(1)$\Rightarrow$(2) Assume that $Q$ is a uniformly 2-absorbing primary ideal
of $V$. Let $xyz\in Q$ for some $x,y,z\in K$ such that $%
xy\notin Q$ . If $z\notin V$, then $z^{-1}\in V$, since $V$ is valuation. So 
$xyzz^{-1}=xy\in Q$, a contradiction. Hence $z\in V$. If $x,y\in V$, then
there is nothing to prove. If $y\notin V$, then $xz\in Q\subseteq\sqrt{Q}$,
and if $x\notin V$, then $yz\in Q$. Consequently we have the claim.\newline
(2)$\Rightarrow$(1) Is clear.
\end{proof}

\begin{proposition}
Let $Q_1,~Q_2$ be two Noether strongly primary ideals of ring $R$. Then $%
Q_1\cap Q_2$ and $Q_1Q_2$ are Noether strongly 2-absorbing primary ideals of 
$R$ such that 2-$\mathfrak{e}(Q_1\cap Q_2)\leq max\{\mathfrak{e}(Q_1),%
\mathfrak{e}(Q_2)\}$ and 2-$\mathfrak{e}(Q_1Q_2)\leq \mathfrak{e}(Q_1)+%
\mathfrak{e}(Q_2)$.
\end{proposition}

\begin{proof}
Since $Q_1,~Q_2$ are primary ideals of $R$, then $Q_1\cap Q_2$ and $Q_1Q_2$
are 2-absorbing primary ideals of $R$, by \textrm{\cite[Theorem 2.4]{BTY}}.
\end{proof}

\begin{proposition}\label{uni-abs}
If $Q$ is a uniformly 2-absorbing primary ideal of $R,$ then $Q$ is a
2-absorbing primary ideal of $R$.
\end{proposition}

\begin{proof}
Straightforward.
\end{proof}

\begin{proposition} \label{prop1}
Let $R$ be a ring and $Q$ be a proper ideal of $%
R$.

\begin{enumerate}
\item If $Q$ is a 2-absorbing ideal of $R$, then 

$(a)$ $Q$ is a Noether strongly 2-absorbing primary ideal with 2-$\mathfrak{e}%
(Q)\leq2 $.

$(b)$ $Q$ is a uniformly 2-absorbing primary ideal with 2-$ord(Q)=1$.

\item If $Q$ is a uniformly primary ideal of $R$, then it is a uniformly
2-absorbing primary ideal with 2-$ord(Q)=1$.
\end{enumerate}
\end{proposition}

\begin{proof}
(1) $(a)$ If $Q$ is a 2-absorbing ideal, then it is a 2-absorbing primary
ideal and $(\sqrt{Q})^2\subseteq Q$, by {\cite[Theorem 2.4]{B}}.
 
(b) Is evident.\newline
(2) Let $Q$ be a uniformly primary ideal of $R$ and let $abc\in Q$ for some $a,b,c\in R$
such that $ab\notin Q$ and $ac\notin\sqrt{Q}$. Since $Q$ is primary, $abc\in Q$ and $ac\notin\sqrt{Q}$,
then $b\in Q$. Therefore $bc\in Q$. Consequently $Q$ is a uniformly
2-absorbing primary ideal with 2-$ord(Q)=1$.
\end{proof}

\begin{example}
Let $R=K[X,Y]$ where $K$ is a field. Then $Q=(X^{2},XY,Y^{2})R$ is a
Noether strongly $(X,Y)R$-primary ideal of $R$ and so it is a
Noether strongly 2-absorbing primary ideal of $R$.
\end{example}

\begin{proposition}\label{noe-uni}
\label{noeuni} If $Q$ is a Noether strongly 2-absorbing primary ideal of $R$%
, then $Q$ is a uniformly 2-absorbing primary ideal of $R$ and $\mbox{2-}%
ord(Q)\leq $2-$\mathfrak{e}(Q)$.
\end{proposition}

\begin{proof}
Let $Q$ be a Noether strongly 2-absorbing primary ideal of $R$. Now, let $a,b,c\in R$ such that $abc\in Q$, $ab\notin Q$, $%
ac\notin\sqrt{Q}$. Then $bc\in \sqrt{Q}$ since $Q$
is a 2-absorbing primary ideal of $R$. Thus $(bc)^{2\mbox{-}\mathfrak{e}%
(Q)}\in (\sqrt{Q})^{2\mbox{-}\mathfrak{e}(Q)}\subseteq Q$. Therefore, $Q$ is a
uniformly 2-absorbing primary ideal and also $\mbox{2-}ord(Q)\leq 2\mbox{-}%
\mathfrak{e}(Q)$.
\end{proof}

In the following example, we show that the converse of Proposition \ref{noe-uni} is
not true. We make use of \textrm{\cite[Example 6 and Example 7]{CH}}

\begin{example}
Let $R$ be a ring of characteristic 2 and $T=R[X]$ where $%
X=\{X_{1},X_{2},X_{3},...\}$ is a set of indeterminates over $R$. Let $%
Q=(\{X_{i}^{2}\}_{i=1}^{\infty })T$. By \textrm{\cite[Example 7]{CH}} $Q$ is a
uniformly $P$-primary ideal of $T$ with $ord_T(Q)=1$ where $P=(X)T$. Then $Q$ is a
uniformly 2-absorbing primary ideal of $T$ with 2-$ord_T(Q)=1$, by Proposition \ref{prop1}(2). But $Q$ is not a
Noether strongly 2-absorbing primary ideal since for every positive integer $%
n$, $P^{n}\nsubseteq Q$.
\end{example}

\begin{remark}
Every 2-absorbing ideal of ring $R$ is a uniformly 2-absorbing primary
ideal, but the converse does not necessarily hold. For example, let $p,~q$
be two distinct prime numbers. Then $p^{2}q\mathbb{Z}$ is a 2-absorbing
primary ideal of $\mathbb{Z}$, \textrm{\cite[Corollary 2.12]{BTY}}. On the
other hand $(\sqrt{p^{2}q\mathbb{Z}})^{2}=p^{2}q^{2}\mathbb{Z}\subseteq
p^{2}q\mathbb{Z}$, and so $p^{2}q\mathbb{Z}$ is a Noether strongly
2-absorbing primary ideal of $\mathbb{Z}$. Hence Proposition \ref{noeuni}
implies that $p^{2}q\mathbb{Z}$ is a uniformly 2-absorbing primary ideal.
But, notice that $p^{2}q\in p^{2}q\mathbb{Z}$ and neither $p^{2}\in p^{2}q%
\mathbb{Z}$ nor $pq\in p^{2}q\mathbb{Z}$ which shows that $p^{2}q\mathbb{Z}$
is not a 2-absorbing ideal of $\mathbb{Z}$. Also, it is easy to see that $%
p^{2}q\mathbb{Z}$ is not primary and so it is not a uniformly primary ideal
of $\mathbb{Z}$. Consequently the two concepts of uniformly primary ideals
and of uniformly 2-absorbing primary ideals are different in general.
\end{remark}

\begin{proposition}
\label{rad} Let $R$ be a ring and $Q$ be a proper ideal of $R$. If $Q$ is a
uniformly 2-absorbing primary ideal of $R$, then one of the following
conditions must hold:

\begin{enumerate}
\item $\sqrt{Q}=\mathfrak{p}$ is a prime ideal.

\item $\sqrt{Q}=\mathfrak{p}\cap\mathfrak{q}$, where $\mathfrak{p}$ and $%
\mathfrak{q}$ are the only distinct prime ideals of $R$ that are minimal
over $Q$.
\end{enumerate}
\end{proposition}

\begin{proof}
Use {\cite[Theorem 2.3]{BTY}}.
\end{proof}


Let $R$ be a ring and $I$ be an ideal of $R$. We denote by $I^{[n]}$ the
ideal of $R$ generated by the $n$-th powers of all elements of $I$. If $n!$
is a unit in $R$, then $I^{[n]}=I^n$, see \cite{AKL}.

\begin{theorem}
Let $Q$ be a proper ideal of $R$. Then the following conditions are
equivalent:

\begin{enumerate}
\item $Q$ is uniformly primary;

\item There exists a positive integer $n$ such that for every ideals $I,J$
of $R$, $IJ\subseteq Q$ implies that either $I\subseteq Q$ or $%
J^{[n]}\subseteq Q$;

\item There exists a positive integer $n$ such that for every $a\in R$
either $a\in Q$ or $(Q:_Ra)^{[n]}\subseteq Q$;

\item There exists a positive integer $n$ such that for every $a\in R$
either $a^n\in Q$ or $(Q:_Ra)=Q$.
\end{enumerate}
\end{theorem}

\begin{proof}
(1)$\Rightarrow $(2) Suppose that $Q$ is uniformly primary with $ord(Q)=n$.
Let $IJ\subseteq Q$ for some ideals $I,~J$ of $R$. Assume that neither $%
I\subseteq Q$ nor $J^{[n]}\subseteq Q$. Then there exist elements $a\in
I\backslash Q$ and $b^{n}\in J^{[n]}\backslash Q$, where $b\in J$. Since $%
ab\in IJ\subseteq Q$, then either $a\in Q$ or $b^{n}\in Q$, which is a
contradiction. Therefore either $I\subseteq Q$ or $J^{[n]}\subseteq Q$.%
\newline
(2)$\Rightarrow $(3) Note that $a(Q:_{R}a)\subseteq Q$ for every $a\in R$. 
\newline
(3)$\Rightarrow $(1) and (1)$\Leftrightarrow $(4) have easy verifications.
\end{proof}

\begin{corollary}
Let $R$ be a ring. Suppose that $n!$ is a unit in $R$ for every positive
integer $n$, and $Q$ is a proper ideal of $R$. The following conditions are
equivalent:

\begin{enumerate}
\item $Q$ is uniformly primary;

\item There exists a positive integer $n$ such that for every ideals $I,J$
of $R$, $IJ\subseteq Q$ implies that either $I\subseteq Q$ or $%
J^{n}\subseteq Q$;

\item There exists a positive integer $n$ such that for every $a\in R$
either $a\in Q$ or $(Q:_Ra)^{n}\subseteq Q$;

\item There exists a positive integer $n$ such that for every $a\in R$
either $a^n\in Q$ or $(Q:_Ra)=Q$.
\end{enumerate}
\end{corollary}

In the following theorem we characterize uniformly 2-absorbing primary
ideals.

\begin{theorem}
\label{main1} Let $Q$ be a proper ideal of $R$. Then the following
conditions are equivalent:

\begin{enumerate}
\item $Q$ is uniformly 2-absorbing primary;

\item There exists a positive integer $n$ such that for every $a,b\in R$
either $(ab)^n\in Q$ or $(Q:_Rab)\subseteq(Q:_Ra)\cup(\sqrt{Q}:_Rb)$;

\item There exists a positive integer $n$ such that for every $a,b\in R$
either $(ab)^n\in Q$ or $(Q:_Rab)=(Q:_Ra)$ or $(Q:_Rab)\subseteq(\sqrt{Q}%
:_Rb)$;

\item There exists a positive integer $n$ such that for every $a,b\in R$ and
every ideal $I$ of $R$, $abI\subseteq Q$ implies that either $aI\subseteq Q$
or $bI\subseteq\sqrt{Q}$ or $(ab)^n\in Q$;

\item There exists a positive integer $n$ such that for every $a,b\in R$
either $ab\in Q$ or $(Q:_Rab)^{[n]}\subseteq(\sqrt{Q}:_Ra)\cup(Q:_Rb^n)$;

\item There exists a positive integer $n$ such that for every $a,b\in R$
either $ab\in Q$ or $(Q:_Rab)^{[n]}\subseteq(\sqrt{Q}:_Ra)$ or $%
(Q:_Rab)^{[n]}\subseteq(Q:_Rb^n)$.
\end{enumerate}
\end{theorem}

\begin{proof}
(1)$\Rightarrow $(2) Suppose that $Q$ is uniformly 2-absorbing primary with $%
\mbox{2-}ord(Q)=n$. Assume that $a,b\in R$ such that $(ab)^{n}\notin Q$. Let 
$x\in (Q:_{R}ab)$. Thus $xab\in Q$, and so either $xa\in Q$ or $xb\in \sqrt{Q%
}$. Hence $x\in (Q:_{R}a)$ or $x\in (\sqrt{Q}:_{R}b)$ which shows that $%
(Q:_{R}ab)\subseteq (Q:_{R}a)\cup (\sqrt{Q}:_{R}b)$.\newline
(2)$\Rightarrow $(3) By the fact that if an ideal is a subset of the union
of two ideals, then it is a subset of one of them.\newline
(3)$\Rightarrow $(4) Suppose that $n$ is a positive number which exists by
part (3). Let $a,b\in R$ and $I$ be an ideal of $R$ such that $abI\subseteq
Q $ and $(ab)^{n}\notin Q$. Then $I\subseteq (Q:_{R}ab)$, and so $I\subseteq
(Q:_{R}a)$ or $I\subseteq (\sqrt{Q}:_{R}b)$, by (3). Consequently $%
aI\subseteq Q$ or $bI\subseteq \sqrt{Q}$.\newline
(4)$\Rightarrow $(1) Is easy.\newline
(1)$\Rightarrow $(5) Suppose that $Q$ is uniformly 2-absorbing primary with $%
\mbox{2-}ord(Q)=n$. Assume that $a,b\in R$ such that $ab\notin Q$. Let $x\in
(Q:_{R}ab)$. Then $abx\in Q$. So $ax\in\sqrt{Q}$ or $(bx)^n\in Q$. Hence $%
x^n\in(\sqrt{Q}:_Ra)$ or $x^n\in(Q:_Rb^n)$. Consequently $%
(Q:_Rab)^{[n]}\subseteq(\sqrt{Q}:_Ra)\cup(Q:_Rb^n)$;\newline
(5)$\Rightarrow $(6) Is similar to the proof of (2)$\Rightarrow $(3).\newline
(6)$\Rightarrow $(1) Assume (6). Let $abc\in Q$ for some $a,b,c\in R$ such
that $ab\notin Q$. Then $c\in(Q:_Rab)$ and thus $c^n\in(Q:_Rab)^{[n]}$. So,
by part (6) we have that $c^n\in(\sqrt{Q}:_Ra)$ or $c^n\in(Q:_Rb^n)$.
Therefore $ac\in\sqrt{Q}$ or $(bc)^n\in Q$, and so $Q$ is uniformly
2-absorbing primary.
\end{proof}

\begin{corollary}
Let $R$ be a ring. Suppose that $n!$ is a unit in $R$ for every positive
integer $n$, and $Q$ is a proper ideal of $R$. The following conditions are
equivalent:

\begin{enumerate}
\item $Q$ is uniformly 2-absorbing primary;

\item There exists a positive integer $n$ such that for every $a,b\in R$
either $ab\in Q$ or $(Q:_Rab)^{n}\subseteq(\sqrt{Q}:_Ra)\cup(Q:_Rb^n)$;

\item There exists a positive integer $n$ such that for every $a,b\in R$
either $ab\in Q$ or $(Q:_Rab)^{n}\subseteq(\sqrt{Q}:_Ra)$ or $%
(Q:_Rab)^{n}\subseteq(Q:_Rb^n)$.
\end{enumerate}
\end{corollary}

\begin{proposition}
Let $Q$ be a uniformly 2-absorbing primary ideal of $R$ and $x\in
R\backslash Q$ be idempotent. The following conditions hold:

\begin{enumerate}
\item $(\sqrt{Q}:_Rx)=\sqrt{(Q:_Rx)}$.

\item $(Q:_Rx)$ is a uniformly 2-absorbing primary ideal of $R$ with 2-$%
ord((Q:_Rx))\leq$ 2-$ord(Q)$.
\end{enumerate}
\end{proposition}

\begin{proof}
(1) Is easy.

(2) Suppose that 2-$ord(Q)=n$. Let $abc\in(Q:_Rx)$ for some $a,b,c\in R$.
Then $a(bc)x\in Q$ and so either $abc\in Q$ or $ax\in \sqrt{Q}$ or $%
(bc)^nx\in Q$. If $abc\in Q$, then either $ab\in Q\subseteq{(Q:_Rx)}$ or $%
ac\in\sqrt{Q}\subseteq\sqrt{(Q:_Rx)}$ or $(bc)^n\in Q\subseteq{(Q:_Rx)}$. If 
$ax\in \sqrt{Q}$, then $ac\in(\sqrt{Q}:_Rx)=\sqrt{(Q:_Rx)}$ by part (1). In
the third case we have $(bc)^n\in(Q:_Rx)$. Hence $(Q:_Rx)$ is a uniformly
2-absorbing primary ideal of $R$ with 2-$ord((Q:_Rx))\leq n$ .
\end{proof}

\begin{proposition}
Let $I$ be a proper ideal of ring $R$.

\begin{enumerate}
\item $\sqrt{I}$ is a 2-absorbing ideal of $R$.

\item For every $a,b,c\in R$, $abc\in I$ implies that $ab\in \sqrt{I}$ or $%
ac\in \sqrt{I}$ or $bc\in \sqrt{I}$;

\item $\sqrt{I}$ is a 2-absorbing primary ideal of $R$;

\item $\sqrt{I}$ is a Noether 2-absorbing primary ideal of $R$ $($2-$%
\mathfrak{e}(\sqrt{I})=1)$;

\item $\sqrt{I}$ is a uniformly 2-absorbing primary ideal of $R$.
\end{enumerate}
\end{proposition}

\begin{proof}
(1)$\Rightarrow $(2) Is trivial.\newline
(2)$\Rightarrow $(1) Let $xyz\in \sqrt{I}$ for some $x,y,z\in R$. Then there
exists a positive integer $m$ such that $x^{m}y^{m}z^{m}\in I$. So, the
hypothesis in (2) implies that $x^{m}y^{m}\in \sqrt{I}$ or $x^{m}z^{m}\in 
\sqrt{I}$ or $y^{m}z^{m}\in \sqrt{I}$. Hence $xy\in \sqrt{I}$ or $xz\in 
\sqrt{I}$ or $yz\in \sqrt{I}$ which shows that $\sqrt{I}$ is a 2-absorbing
ideal.\newline
(1)$\Leftrightarrow $(3) and (3)$\Rightarrow $(4) are clear. \newline
(4)$\Rightarrow $(5) By Proposition \ref{noeuni}.\newline
(5)$\Rightarrow $(3) Is easy.
\end{proof}



\begin{proposition}
If $Q_1$ is a uniformly $P$-primary ideal of $R$ and $Q_2$ is a uniformly $P$%
-2-absorbing primary ideal of $R$ such that $Q_{1}\subseteq Q_{2}$, then $%
\mbox{2-}ord(Q_{2})\leq ord(Q_{1})$.
\end{proposition}

\begin{proof}
Let $ord(Q_{1})=m$ and $\mbox{2-}ord(Q_{2})=n$. Then there are $a,b,c\in R$
such that $abc\in Q_{2}$, $ab\notin Q_{2}$, $ac\notin\sqrt{Q_{2}}$ and $%
(bc)^n\in Q_2$ but $(bc)^{n-1}\notin Q_2$. Thus $bc\in\sqrt{Q_{2}}=\sqrt{Q_1}
$. Hence $(bc)^{m}\in Q_{1}\subseteq Q_{2}$ by {\cite[Proposition 8]{CH}}.
Therefore, $n>m-1$ and so $n\geq m$.
\end{proof}

\begin{theorem}
Let $R$ be a ring and $\{Q_{i}\}_{i\in I}$ be a chain of uniformly $P$%
-2-absorbing primary ideals such that $\max_{i\in I}\{\mbox{2-}%
ord(Q_{i})\}=n,$ where $n$ is a positive integer. Then $Q=\bigcap\limits_{i%
\in I}Q_{i}$ is a uniformly $P$-2-absorbing primary ideal of $R$ with 2-$%
ord(Q)\leq n$.
\end{theorem}

\begin{proof}
It is clear that $\sqrt{Q}=\bigcap\limits_{i\in I}\sqrt{Q_{i}}=P$. Let $%
a,b,c\in R$ such that $abc\in Q,$ $ab\notin Q$ and $(bc)^n\notin Q$. Since $%
\{Q_{i}\}_{i\in I}$ is a chain, there exists some $k\in I$ such that $%
ab\notin Q_k$ and $(bc)^n\notin Q_k$. On the other hand $Q_k$ is uniformly
2-absorbing primary with 2-$ord(Q_k)\leq n$, thus $ac\in\sqrt{Q_i}=\sqrt{Q}$%
, and so $Q$ is a uniformly 2-absorbing primary ideal of $R$ with 2-$%
ord(Q)\leq n$.
\end{proof}

In the following remark, we show that if $Q_{1}$ and $Q_{2}$ are uniformly
2-absorbing primary ideals of $R$, then $Q_{1}\cap Q_{2}$ need not be a
uniformly 2-absorbing primary ideal of $R$.

\begin{remark}
Let $p,q,r$
be distinct prime numbers. Then $p^{2}q\mathbb{Z}$ and $r\mathbb{Z}$ are uniformly 2-absorbing primary ideals of $\mathbb{Z}$.
Notice that $p^{2}qr\in p^{2}q\mathbb{Z}\cap r\mathbb{Z}$ and neither $p^{2}q\in p^{2}q%
\mathbb{Z}\cap r\mathbb{Z}$ nor $p^2r\in \sqrt{p^{2}q\mathbb{Z}\cap r\mathbb{Z}}=p\mathbb{Z}\cap q\mathbb{Z}\cap r\mathbb{Z}$
nor $qr\in \sqrt{p^{2}q\mathbb{Z}\cap r\mathbb{Z}}=p\mathbb{Z}\cap q\mathbb{Z}\cap r\mathbb{Z}$. Hence $p^{2}q\mathbb{Z}\cap r\mathbb{Z}$
is not a 2-absorbing primary ideal of $\mathbb{Z}$ which shows that it is not a uniformly 2-absorbing primary ideal of $\mathbb{Z}$.
\end{remark}


\begin{theorem}\label{intersection} 
Let $Q_1,~Q_2$ be uniformly primary ideals
of ring $R$.

\begin{enumerate}
\item $Q_1\cap Q_2$ is a uniformly 2-absorbing primary ideal of $R$ with 2-$%
ord(Q_1\cap Q_2)\leq\mbox{max}\{ord(Q_1),ord(Q_2)\}$.

\item $Q_1Q_2$ is a uniformly 2-absorbing primary ideal of $R$ with 2-$%
ord(Q_1 Q_2)\leq ord(Q_1)+ord(Q_2)$.
\end{enumerate}
\end{theorem}

\begin{proof}
(1) Let $Q_1,~Q_2$ be uniformly primary. Set $n=\mbox{max}%
\{ord(Q_1),ord(Q_2)\}$. Assume that for some $a,b,c\in R$, $abc\in Q_1\cap
Q_2$, $ab\notin{Q_1\cap Q_2}$ and $ac\notin\sqrt{Q_1\cap Q_2}$. Since $Q_1$
and $Q_2$ are primary ideals of $R$, then $Q_1\cap Q_2$ is 2-absorbing
primary by \textrm{\cite[Theorem 2.4]{BTY}}. Therefore $bc\in\sqrt{Q_1\cap
Q_2}=\sqrt{Q_1}\cap\sqrt{Q_2}$. By \textrm{\cite[Proposition 8]{CH}} we have
that $(bc)^{ord(Q_1)}\in Q_1$ and $(bc)^{ord(Q_2)}\in Q_2$. Hence $(bc)^n\in
Q_1\cap Q_2$ which shows that $Q_1\cap Q_2$ is uniformly 2-absorbing primary
and 2-$ord(Q_1\cap Q_2)\leq n$.\newline
(2) Similar to the proof in (1).
\end{proof}

We recall from \cite{Hoc}, if $R$ is an integral domain and $P$ is a prime ideal of $R$ that can be generated by a
regular sequence of $R$. Then, for each positive integer $n$, the ideal $P^n$
is a $P$-primary ideal of $R$.

\begin{lemma}
$(${\cite[Corollary 4]{CH}}$)$\label{lemch} Let $R$ be a ring and $P$ be a
prime ideal of $R$. If $P^n$ is a $P$-primary ideal of $R$ for some positive
integer $n$, then $P^n$ is a uniformly primary ideal of $R$ with $%
ord(P^n)\leq n$.
\end{lemma}

\begin{corollary}
Let $R$ be a ring and $P_1,~P_2$ be prime ideals of $R$. If $P_1^n$ is a $%
P_1 $-primary ideal of $R$ for some positive integer $n$ and $P_2^m$ is a $%
P_2$-primary ideal of $R$ for some positive integer $m$, then $P_1^nP_2^m$
and $P_1^n\cap P_2^m$ are uniformly 2-absorbing primary ideals of $R$ with 2-%
$ord(P_1^nP_2^m)\leq n+m$ and 2-$ord(P_1^n\cap P_2^m)\leq max\{n,m\}$.
\end{corollary}

\begin{proof}
By Theorem \ref{intersection} and Lemma \ref{lemch}.
\end{proof}

\begin{proposition}\label{monoepi} 
Let $f:R\longrightarrow R^{\prime }$ be a homomorphism of
commutative rings. Then the following statements hold:

\begin{enumerate}
\item If $Q^{\prime }$ is a uniformly 2-absorbing primary ideal of $%
R^{\prime }$, then $f^{-1}(Q^{\prime })$ is a uniformly 2-absorbing primary
ideal of $R$ with $\mbox{2-}ord_{R}(f^{-1}(Q^{\prime }))\leq \mbox{2-}%
ord_{R^{\prime }}(Q^{\prime })$.

\item If $f$ is an epimorphism and $Q$ is a uniformly 2-absorbing primary
ideal of $R$ containing $\ker (f)$, then $f(Q)$ is a uniformly 2-absorbing
primary ideal of $R^{\prime }$ with $\mbox{2-}ord_{R^{\prime }}(f(Q))\leq\mbox{2-}ord_{R}(Q) $.
\end{enumerate}
\end{proposition}

\begin{proof}
(1) Set $N=\mbox{2-}ord_{R^{\prime }}(Q^{\prime })$. Let $a,b,c\in R$ such that $abc\in $ $f^{-1}(Q^{\prime })$, $ab\notin $ $%
f^{-1}(Q^{\prime })$ and $ac\notin\sqrt{f^{-1}(Q^{\prime })}=f^{-1}(\sqrt{Q^{\prime }})$. 
Then $f(abc)=f(a)f(b)f(c)\in Q^{\prime }$, $f(ab)=f(a)f(b)
\notin Q^{\prime }$ and  $f(ac)=f(a)f(c)
\notin \sqrt{Q^{\prime }}$. Since $Q^{\prime }$ is
a uniformly 2-absorbing primary ideal of $R^{\prime }$, then $%
f^{N}(bc)\in Q^{\prime }.$ Then $f((bc)^{N})\in Q^{\prime }$ and so $(bc)^{N}\in f^{-1}(Q^{\prime })
$. Thus $f^{-1}(Q^{\prime })$ is a uniformly 2-absorbing primary ideal of $R$ with $\mbox{2-}ord_{R}(f^{-1}(Q^{\prime }))\leq N=\mbox{2-}%
ord_{R^{\prime }}(Q^{\prime })$.

(2) Set $N=\mbox{2-}ord_{R}(Q)$. Let $a,b,c\in R^{\prime }$ such that $abc\in $ $f(Q)$, $ab\notin $ $f(Q)$
and $ac\notin\sqrt{f(Q)}$. Since $f$ is an
epimorphism, then there exist $x,y,z\in R$ such that $f(x)=a$, $f(y)=b$ and $%
f(z)=c.$ Then $f(xyz)=abc$ $\in $ $f(Q)$, $f(xy)=ab\notin $ $f(Q)$ and $%
f(xz)=ac\notin\sqrt{f(Q)}$. Since $\ker (f)\subseteq Q$, then $xyz\in
Q$. Also $xy\notin Q$, and $xz\notin\sqrt{Q}$, since $f(\sqrt{Q})\subseteq\sqrt{f(Q)}$. Then $%
(yz)^{N}\in Q$ since $Q$ is a uniformly 2-absorbing primary ideal of $%
R$. Thus $f((yz)^{N})=(f(y)f(z))^{N}=(bc)^{N}\in f(Q)$.
Therefore, $f(Q)$ is a uniformly 2-absorbing primary ideal of $R^{\prime }.$
Moreover $\mbox{2-}ord_{R^{\prime }}(f(Q))\leq N=\mbox{2-}ord_{R}(Q) $.
\end{proof}

As an immediate consequence of Proposition \ref{monoepi} we have the
following result:

\begin{corollary}
\label{frac} Let $R$ be a ring and $Q$ be an ideal of $R$.

\begin{enumerate}
\item If $R^\prime$ is a subring of $R$ and $Q$ is a uniformly 2-absorbing
primary ideal of $R$, then $Q\cap R^\prime$ is a uniformly 2-absorbing
primary ideal of $R^\prime$ with $\mbox{2-}ord_{R^\prime}(Q\cap R^\prime)\leq \mbox{2-}ord_{R}(Q) $.

\item Let $I$ be an ideal of $R$ with $I\subseteq Q$. Then $Q$ is a
uniformly 2-absorbing primary ideal of $R$ if and only if $Q/I$ is a
uniformly 2-absorbing primary ideal of $R/I$.
\end{enumerate}
\end{corollary}

\begin{corollary}
Let $Q$ be an ideal of ring $R$. Then $\langle Q,X\rangle$ is a uniformly
2-absorbing primary ideal of $R[X]$ if and only if $Q$ is a uniformly
2-absorbing primary ideal of $R$.
\end{corollary}

\begin{proof}
By Corollary \ref{frac}(2) and regarding the isomorphism $\langle
Q,X\rangle/\langle X\rangle\simeq Q$ in $R[X]/\langle X\rangle\simeq R$ we
have the result.
\end{proof}

\begin{corollary}
Let $R$ be a ring, $Q$ a proper ideal of $R$ and $X=\{X_{i}\}_{i\in I}$ a
collection of indeterminates over $R$. If $QR[X]$ is a uniformly 2-absorbing
primary ideal of $R[X]$, then $Q$ is a uniformly 2-absorbing primary ideal
of $R$ with $\mbox{2-}ord_R(Q)\leq \mbox{2-}ord_{R[X]}(QR[X])$.
\end{corollary}

\begin{proof}
It is clear from Corollary \ref{frac}(1).
\end{proof}

\begin{proposition}\label{multi}
Let $S$ be a multiplicatively closed subset of $R$ and $Q$ be a proper ideal
of $R$. Then the following conditions hold:
\begin{enumerate}
\item If $Q$ is a uniformly 2-absorbing primary ideal of $R$ such that $%
Q\cap S=\emptyset $, then $S^{-1}Q$ is a uniformly 2-absorbing primary ideal
of $S^{-1}R$ with $\mbox{2-}ord(S^{-1}Q)\leq \mbox{2-}ord(Q)$.

\item If $S^{-1}Q$ is a uniformly 2-absorbing primary ideal of $S^{-1}R$ and 
$S\cap Z_{Q}(R)=\emptyset $, then $Q$ is a uniformly 2-absorbing primary
ideal of $R$ with $\mbox{2-}ord(Q)\leq \mbox{2-}ord(S^{-1}Q).$
\end{enumerate}
\end{proposition}

\begin{proof}
(1) Set $N:=\mbox{2-}ord(Q)$. Let $a,b,c\in R$ and $s,t,k\in S$ such that $\frac{a}{s}\frac{b}{t}\frac{c}{k%
}\in S^{-1}Q$, $\frac{a}{s}\frac{b}{t}\notin S^{-1}Q$, $\frac{a}{s}\frac{c}{%
k}\notin \sqrt{S^{-1}Q}=S^{-1}\sqrt{Q}$. Thus there is $u\in S$ such that $%
uabc\in Q$. By assumptions we have that $uab\notin Q$ and $uac\notin \sqrt{Q}$. 
Since $Q$ is a uniformly 2-absorbing primary ideal of $R$, then $(bc)^{N}\in Q$. Hence $(\frac{b}{t}\frac{c
}{k})^{N}\in S^{-1}Q$. Consequently, $S^{-1}Q$ is a uniformly 2-absorbing
primary ideal of $S^{-1}R$ and $\mbox{2-}ord(S^{-1}Q)\leq N= \mbox{2-}ord(Q)$.

(2) Set $N:=\mbox{2-}ord(S^{-1}Q)$. Let $a,b,c\in R$ such that $abc\in $ $Q$, $ab\notin $ $Q$ and $%
ac\notin\sqrt{Q}$. Then $\frac{abc}{1}=\frac{%
a}{1}\frac{b}{1}\frac{c}{1}\in S^{-1}Q$, $\frac{ab}{1}=\frac{a}{1}\frac{b}{1}%
\notin $ $S^{-1}Q$ and $\frac{ac}{1}=\frac{a}{1}\frac{c}{1}%
\notin  \sqrt{S^{-1}Q}=S^{-1}\sqrt{Q}$, because $S\cap
Z_{Q}(R)=\emptyset$ and $S\cap Z_{\sqrt{Q}}(R)=\emptyset$. Since $S^{-1}Q$ is a uniformly 2-absorbing primary
ideal of $S^{-1}R$, then $(\frac{b}{1}\frac{c}{1})^{N}=\frac{%
(bc)^{N}}{1}\in S^{-1}Q$. Then there exists $u\in S$ such that $%
u(bc)^{N}\in Q$. Hence $(bc)^{N}\in Q$ because $S\cap Z_{Q}(R)=\emptyset $. Thus $Q$ is a uniformly 2-absorbing
primary ideal of $R$ and $\mbox{2-}ord(Q)\leq N=\mbox{2-}ord(S^{-1}Q).$
\end{proof}

\begin{proposition}\label{abs-noe-uni}
Let $Q$ be a 2-absorbing primary ideal of ring $R$ and $P=\sqrt{Q}$ be a
finitely generated ideal of $R$. Then $Q$ is a Noether strongly 2-absorbing
primary ideal of $R$. Thus $Q$ is a uniformly 2-absorbing primary ideal of $%
R $.
\end{proposition}

\begin{proof}
It is clear from \textrm{\cite[Lemma 8.21]{Sh}} and Proposition \ref{noe-uni}.
\end{proof}

\begin{corollary}
Let $R$ be a Noetherian ring and $Q$ a proper ideal of $R$. Then the
following conditions are equivalent:
\end{corollary}

\begin{enumerate}
\item $Q$ is a uniformly 2-absorbing primary ideal of $R;$

\item $Q$ is a Noether strongly 2-absorbing primary ideal of $R$;

\item $Q$ is a 2-absorbing primary ideal of $R$.
\end{enumerate}

\begin{proof}
Apply Proposition \ref{noe-uni} and Proposition \ref{abs-noe-uni}.
\end{proof}

We recall from \cite{Huk} the construction of idealization of a module. Let $%
R$ be a ring and $M$ be an $R$-module. Then $R(+)M = R \times M$ is a ring
with identity $(1, 0)$ under addition defined by $(r, m) + (s, n) = (r + s,
m + n)$ and multiplication defined by $(r, m)(s, n) = (rs, rn + sm)$. Note
that $\sqrt{I}(+)M=\sqrt{I(+)M}$.

\begin{proposition}
Let $R$ be a ring, $Q$ be a proper ideal of $R$ and $M$ be an $R$-module.
The following conditions are equivalent:

\begin{enumerate}
\item $Q(+)M$ is a uniformly 2-absorbing primary ideal of $R(+)M$;

\item $Q$ is a uniformly 2-absorbing primary ideal of $R$.
\end{enumerate}
\end{proposition}

\begin{proof}
The proof is routine.
\end{proof}

\begin{theorem}\label{product1} 
Let $R=R_1\times R_2$, where $R_1$ and $R_2$
are rings with $1\neq0$. Let $Q$ be a proper ideal of $R$. Then the
following conditions are equivalent:

\begin{enumerate}
\item $Q$ is a uniformly 2-absorbing primary ideal of $R$;

\item Either $Q=Q_1\times R_2$ for some uniformly 2-absorbing primary ideal $%
Q_1$ of $R_1$ or $Q=R_1\times Q_2$ for some uniformly 2-absorbing primary
ideal $Q_2$ of $R_2$ or $Q=Q_1\times Q_2$ for some uniformly primary ideal $%
Q_1$ of $R_1$ and some uniformly primary ideal $Q_2$ of $R_2$.
\end{enumerate}
\end{theorem}

\begin{proof}
(1)$\Rightarrow$(2) Assume that $Q$ is a uniformly 2-absorbing primary ideal
of $R$ with 2-$ord_R(Q)=n$. We know that $Q$ is in the form of $Q_1\times
Q_2 $ for some ideal $Q_1$ of $R_1$ and some ideal $Q_2$ of $R_2$. Suppose
that $Q_2=R_2$. Since $Q$ is a proper ideal of $R$, $Q_1\neq R_1$. Let $%
R^{\prime }=\frac{R}{\{0\}\times R_2}$. Then $Q^{\prime }=\frac{Q}{%
\{0\}\times R_2}$ is a uniformly 2-absorbing primary ideal of $R^{\prime }$
by Corollary \ref{frac}(2). Since $R^{\prime }$ is ring-isomorphic to $R_1$
and $Q_1\simeq Q^{\prime }$, $Q_1$ is a uniformly 2-absorbing primary ideal
of $R_1$. Suppose that $Q_1=R_1$. Since $Q$ is a proper ideal of $R$, $%
Q_2\neq R_2$. By a similar argument as in the previous case, $Q_2$ is a
uniformly 2-absorbing primary ideal of $R_2$. Hence assume that $Q_1\neq R_1$
and $Q_2\neq R_2$. We claim that $Q_1$ is a uniformly primary ideal of $R_1$%
. Assume that $x,~y\in R_1$ such that $xy\in Q_1$ but $x\notin Q_1$. Notice
that $(x,1)(1,0)(y,1)=(xy,0)\in Q$, but neither $(x,1)(1,0)=(x,0)\in Q$ nor $%
(x,1)(y,1)=(xy,1)\in \sqrt{Q}$. So $[(1,0)(y,1)]^n=(y^n,0)\in Q$. Therefore $%
y^n\in Q_1$. Thus $Q_1$ is a uniformly primary ideal of $R_1$ with $%
ord_{R_1}(Q_1)\leq n$. Now, we claim that $Q_2$ is a uniformly primary ideal
of $R_2$. Suppose that for some $z,~w\in R_2$, $zw\in Q_2$ but $z\notin Q_2$%
. Notice that $(1,z)(0,1)(1,w)=(0,zw)\in Q$, but neither $%
(1,z)(0,1)=(0,z)\in Q$ nor $(1,z)(1,w)=(1,zw)\in \sqrt{Q}$. Therefore $%
[(0,1)(1,w)]^n=(0,w^n)\in Q$, and so $w^n\in Q_2$ which shows that $Q_2$ is
a uniformly primary ideal of $R_2$ with $ord_{R_2}(Q_2)\leq n$. Consequently
when $Q_1\neq R_1$ and $Q_2\neq R_2$ we have that $\mbox{max}%
\{ord_{R_1}(Q_1),ord_{R_2}(Q_2)\}\leq\mbox{2-}ord_R(Q)$.\newline
(2)$\Rightarrow$(1) If $Q=Q_1\times R_2$ for some uniformly 2-absorbing
primary ideal $Q_1$ of $R_1$, or $Q=R_1\times Q_2$ for some uniformly
2-absorbing primary ideal $Q_2$ of $R_2$, then it is clear that $Q$ is a
uniformly 2-absorbing primary ideal of $R$. Hence assume that $Q=Q_1\times
Q_2$ for some uniformly primary ideal $Q_1$ of $R_1$ and some uniformly
primary ideal $Q_2$ of $R_2$. Then $Q^{\prime }_1=Q_1\times R_2$ and $%
Q^{\prime }_2=R_1\times Q_2$ are uniformly primary ideals of $R$ with $%
ord_{R}(Q^{\prime }_1)\leq ord_{R_1}(Q_1)$ and $ord_{R}(Q^{\prime }_2)\leq
ord_{R_2}(Q_2)$. Hence $Q^{\prime }_1\cap Q^{\prime }_2=Q_1\times Q_2=Q$ is
a uniformly 2-absorbing primary ideal of $R$ with 2-$ord_R(Q)\leq\mbox{max}%
\{ord_{R_1}(Q_1),ord_{R_2}(Q_2)\}$ by Theorem \ref{intersection}.
\end{proof}

\begin{lemma}\label{product2} 
Let $R=R_1\times R_2\times\cdots\times R_n$, where $R_1,R_2,
...,R_n$ are rings with $1\neq0$. A proper ideal $Q$ of $R$ is a uniformly
primary ideal of $R$ if and only if $Q=\times_{i=1}^{n}Q_i$ such that for
some $k\in\{1,2,...,n\}$, $Q_k$ is a uniformly primary ideal of $R_k$, and $%
Q_i=R_i$ for every $i\in\{1,2,...,n\}\backslash\{k\}$.
\end{lemma}

\begin{proof}
($\Rightarrow$) Let $Q$ be a uniformly primary ideal of $R$ with $ord_R(Q)=m$%
. We know $Q=\times_{i=1}^{n}Q_i$ where for every $1\leq i\leq n$, $Q_i$ is
an ideal of $R_i$, respectively. Assume that $Q_r$ is a proper ideal of $R_r$
and $Q_s$ is a proper ideal of $R_s$ for some $1\leq r<s\leq n$. Since 
\begin{equation*}
(0,\dots,0,\overbrace{1_{R_r}}^{r\mbox{-th}},0,\dots,0)(0,\dots,0,\overbrace{%
1_{R_s}}^{s\mbox{-th}},0,\dots,0)=(0,\dots,0)\in Q,
\end{equation*}
then either $(0,\dots,0,\overbrace{1_{R_r}}^{r\mbox{-th}},0,\dots,0)\in Q$
or $(0,\dots,0,\overbrace{1_{R_s}}^{s\mbox{-th}},0,\dots,0)^m\in Q$, which
is a contradiction. Hence exactly one of the $Q_i$'s is proper, say $Q_k$.
Now, we show that $Q_k$ is a uniformly primary ideal of $R_k$. Let $ab\in
Q_k $ for some $a,b\in R_k$ such that $a\notin Q_k$. Therefore 
\begin{equation*}
(0,\dots,0,\overbrace{a}^{k\mbox{-th}},0,\dots,0)(0,\dots,0,\overbrace{b}^{k%
\mbox{-th}},0,\dots,0)=(0,\dots,0,\overbrace{ab}^{k\mbox{-th}},0,\dots,0)\in
Q,
\end{equation*}
but $(0,\dots,0,\overbrace{a}^{k\mbox{-th}},0,\dots,0)\notin Q$, and so $%
(0,\dots,0,\overbrace{b}^{k\mbox{-th}},0,\dots,0)^m\in Q$. Thus $b^m\in Q_k$
which implies that $Q_k$ is a uniformly primary ideals of $R_k$ with $%
ord_{R_k}(Q_k)\newline
\leq m$.\newline
($\Leftarrow$) Is easy.
\end{proof}

\begin{theorem}\label{product3}
Let $R=R_1\times R_2\times\cdots\times R_n$, where $2\leq n<\infty$, and $%
R_1,R_2, ...,R_n$ are rings with $1\neq0$. For a proper ideal $Q$ of $R$ the
following conditions are equivalent:

\begin{enumerate}
\item $Q$ is a uniformly 2-absorbing primary ideal of $R$.

\item Either $Q=\times^n_{t=1}Q_t$ such that for some $k\in\{1,2,...,n\}$, $%
Q_k$ is a uniformly 2-absorbing primary ideal of $R_k$, and $Q_t=R_t$ for
every $t\in\{1,2,...,n\}\backslash\{k\}$ or $Q=\times_{t=1}^nQ_t$ such that
for some $k,m\in\{1,2,...,n\}$, $Q_k$ is a uniformly primary ideal of $R_k$, 
$Q_m$ is a uniformly primary ideal of $R_m$, and $Q_t=R_t$ for every $%
t\in\{1,2,...,n\}\backslash\{k,m\}$.
\end{enumerate}
\end{theorem}

\begin{proof}
We use induction on $n$. For $n=2$ the result holds by Theorem \ref{product1}
. Then let $3\leq n <\infty$ and suppose that the result is valid when $%
K=R_1\times\cdots\times R_{n-1}$. We show that the result holds when $%
R=K\times R_n$. By Theorem \ref{product1}, $Q$ is a uniformly 2-absorbing
primary ideal of $R$ if and only if either $Q=L\times R_n$ for some
uniformly 2-absorbing primary ideal $L$ of $K$ or $Q=K\times L_n$ for some
uniformly 2-absorbing primary ideal $L_n$ of $R_n$ or $Q=L\times L_n$ for
some uniformly primary ideal $L$ of $K$ and some uniformly primary ideal $%
L_n $ of $R_n$. Notice that by Lemma \ref{product2}, a proper ideal $L$ of $K$
is a uniformly primary ideal of $K$ if and only if $L=\times_{t=1}^{n-1}Q_t$
such that for some $k\in\{1,2,...,n-1\}$, $Q_k$ is a uniformly primary ideal
of $R_k$, and $Q_t=R_t$ for every $t\in\{1,2,...,n-1\}\backslash\{k\}$.
Consequently we reach the claim.
\end{proof}

\section{Special 2-absorbing primary ideals}
\begin{definition}
We say that a proper ideal $Q$ of a ring $R$ is \textit{special 2-absorbing
primary} if it is uniformly 2-absorbing primary with 2-$ord(Q)=1$.
\end{definition}

\begin{remark}
By Proposition \ref{prop1}(2), every primary ideal is a special 2-absorbing
primary ideal. But the converse is not true in general. For example, let $%
p,~q$ be two distinct prime numbers. Then $pq\mathbb{Z}$ is a 2-absorbing
ideal of $\mathbb{Z}$ and so it is a special 2-absorbing primary ideal of $%
\mathbb{Z}$, by Proposition \ref{prop1}(1). Clearly $pq\mathbb{Z}$ is not
primary.
\end{remark}

Recall that a prime ideal $\mathfrak{p}$ of $R$ is called \textit{divided
prime} if $\mathfrak{p}\subset xR$ for every $x\in R\backslash \mathfrak{p}$.

\begin{proposition}
Let $Q$ be a special 2-absorbing primary ideal of $R$ such that $\sqrt{Q}=%
\mathfrak{p}$ is a divided prime ideal of $R$. Then $Q$ is a $\mathfrak{p}$%
-primary ideal of $R$.
\end{proposition}

\begin{proof}
Let $xy\in Q$ for some $x,y\in R$ such that $y\notin\mathfrak{p}$. Then $x\in%
\mathfrak{p}$. Since $\mathfrak{p}$ is a divided prime ideal, $\mathfrak{p}%
\subset yR$ and so there exists $r\in R$ such that $x=ry$. Hence $xy=ry^2\in
Q$. Since $Q$ is special 2-absorbing primary and $y\notin\mathfrak{p}$, then 
$x=ry\in Q$. Consequently $Q$ is a $\mathfrak{p}$-primary ideal of $R$.
\end{proof}

\begin{remark}
Let $p,~q$ be distinct prime numbers. Then by \cite[Theorem 2.4]{BTY} we can
deduce that $p\mathbb{Z}\cap q^2\mathbb{Z}$ is a 2-absorbing primary ideal
of $\mathbb{Z}$. Since $pq^2\in p\mathbb{Z}\cap q^2\mathbb{Z}$, $pq\notin p%
\mathbb{Z}\cap q^2\mathbb{Z}$ and $q^2\notin p\mathbb{Z}\cap q\mathbb{Z}$,
then $p\mathbb{Z}\cap q^2\mathbb{Z}$ is not a special 2-absorbing primary
ideal of $\mathbb{Z}$.
\end{remark}

Notice that for $n=1$ we have that $I^{[n]}=I$.

\begin{theorem}\label{special}
Let $Q$ be a proper ideal of $R$. Then the following conditions are
equivalent:

\begin{enumerate}
\item $Q$ is special 2-absorbing primary;

\item For every $a,b\in R$ either $ab\in Q$ or $(Q:_Rab)=(Q:_Ra)$ or $%
(Q:_Rab)\subseteq(\sqrt{Q}:_Rb)$;

\item For every $a,b\in R$ and every ideal $I$ of $R$, $abI\subseteq Q$
implies that either $ab\in Q$ or $aI\subseteq Q$ or $bI\subseteq\sqrt{Q}$;

\item For every $a\in R$ and every ideal $I$ of $R$ either $aI\subseteq Q$
or $(Q:_RaI)\subseteq(Q:_Ra)\cup(\sqrt{Q}:_RI)$;

\item For every $a\in R$ and every ideal $I$ of $R$ either $aI\subseteq Q$
or $(Q:_RaI)=(Q:_Ra)$ or $(Q:_RaI)\subseteq(\sqrt{Q}:_RI)$;

\item For every $a\in R$ and every ideals $I,J$ of $R$, $aIJ\subseteq Q$
implies that either $aI\subseteq Q$ or $IJ\subseteq \sqrt{Q}$ or $%
aJ\subseteq Q$;

\item For every ideals $I,J$ of $R$ either $IJ\subseteq\sqrt{Q}$ or $%
(Q:_RIJ)\subseteq(Q:_RI)\cup(Q:_RJ)$;

\item For every ideals $I,J$ of $R$ either $IJ\subseteq\sqrt{Q}$ or $%
(Q:_RIJ)=(Q:_RI)$ or $(Q:_RIJ)=(Q:_RJ)$;

\item For every ideals $I,J,K$ of $R$, $IJK\subseteq Q$ implies that either $%
IJ\subseteq \sqrt{Q}$ or $IK\subseteq Q$ or $JK\subseteq Q$.
\end{enumerate}
\end{theorem}

\begin{proof}
(1)$\Leftrightarrow $(2)$\Leftrightarrow $(3) By Theorem \ref{main1}.\newline
(3)$\Rightarrow $(4) Let $a\in R$ and $I$ be an ideal of $R$ such that $%
aI\nsubseteq Q$. Suppose that $x\in (Q:_{R}aI)$. Then $axI\subseteq Q$, and
so by part (3) we have that $x\in (Q:_{R}a)$ or $x\in (\sqrt{Q}:_{R}I)$.
Therefore $(Q:_{R}aI)\subseteq (Q:_{R}a)\cup (\sqrt{Q}:_{R}I)$.\newline
(4)$\Rightarrow $(5)$\Rightarrow $(6)$\Rightarrow $(7)$\Rightarrow $(8)$%
\Rightarrow $(9)$\Rightarrow $(1) Have straightforward proofs.
\end{proof}

\begin{theorem}\label{main2} 
Let $Q$ be a special 2-absorbing primary ideal of $R$ and $%
x\in R\backslash \sqrt{Q}$. The following conditions hold:

\begin{enumerate}
\item $(Q:_Rx)=(Q:_Rx^n)$ for every $n\geq2$.

\item $(\sqrt{Q}:_Rx)=\sqrt{(Q:_Rx)}$.

\item $(Q:_Rx)$ is a special 2-absorbing primary ideal of $R$.
\end{enumerate}
\end{theorem}

\begin{proof}
(1) Clearly $(Q:_Rx)\subseteq(Q:_Rx^n)$ for every $n\geq2$. For the converse
inclusion we use induction on $n$. First we get $n=2$. Let $r\in(Q:_Rx^2)$.
Then $rx^2\in Q$, and so either $rx\in Q$ or $x^2\in\sqrt{Q}$. Notice that $%
x^2\in\sqrt{Q}$ implies that $x\in\sqrt{Q}$ which is a contradiction.
Therefore $rx\in Q$ and so $r\in(Q:_Rx)$. Therefore $(Q:_Rx)=(Q:_Rx^2)$.
Now, assume $n>2$ and suppose that the claim holds for $n-1$, i.e. $%
(Q:_Rx)=(Q:_Rx^{n-1})$. Let $r\in(Q:_Rx^n)$. Then $rx^n\in Q$. Since $x\notin%
\sqrt{Q}$, then we have either $rx^{n-1}\in Q$ or $rx\in Q$. Both two cases
implies that $r\in(Q:_Rx)$. Consequently $(Q:_Rx)=(Q:_Rx^n)$.

(2) It is easy to investigate that $\sqrt{(Q:_Rx)}\subseteq(\sqrt{Q}:_Rx)$.
Let $r\in(\sqrt{Q}:_Rx)$. Then there exists a positive integer $m$ such that 
$(rx)^m\in Q$. So, by part (1) we have that $r^m\in(Q:_Rx)$. Hence $r\in%
\sqrt{(Q:_Rx)}$. Thus $(\sqrt{Q}:_Rx)=\sqrt{(Q:_Rx)}$.

(3) Let $abc\in(Q:_Rx)$ for some $a,b,c\in R$. Then $ax(bc)\in Q$ and so $%
ax\in Q$ or $abc\in Q$ or $bcx\in\sqrt{Q}$. In the first case we have $%
ab\in(Q:_Rx)$. If $abc\in Q$, then either $ab\in Q\subseteq{(Q:_Rx)}$ or $%
ac\in Q\subseteq{(Q:_Rx)}$ or $bc\in\sqrt{Q}\subseteq\sqrt{(Q:_Rx)}$. In the
third case we have $bc\in(\sqrt{Q}:_Rx)=\sqrt{(Q:_Rx)}$ by part (2).
Therefore $(Q:_Rx)$ is a special 2-absorbing primary ideal of $R$.
\end{proof}

\begin{theorem}\label{main3} 
Let $Q$ be an irreducible ideal of $R$. Then $Q$ is special
2-absorbing primary if and only if $(Q:_Rx)=(Q:_Rx^2)$ for every $x\in
R\backslash\sqrt{Q}$.
\end{theorem}

\begin{proof}
$(\Rightarrow)$ By Theorem \ref{main2}.\newline
$(\Leftarrow)$ Let $abc\in Q$ for some $a,b,c\in R$ such that neither $ab\in
Q$ nor $ac\in Q$ nor $bc\in\sqrt{Q}$. We search for a contradiction. Since $%
bc\notin\sqrt{Q}$, then $b\notin\sqrt{Q}$. So, by our hypothesis we have $%
(Q:_Rb)=(Q:_Rb^2)$. Let $r\in(Q+Rab)\cap(Q+Rac)$. Then there are $q_1,q_2\in
Q$ and $r_1,r_2\in R$ such that $r=q_1+r_1ab=q_2+r_2ac$. Hence $%
q_1b+r_1ab^2=q_2b+r_2abc\in Q$. Thus $r_1ab^2\in Q$, i.e., $%
r_1a\in(Q:_Rb^2)=(Q:_Rb)$. Therefore $r_1ab\in Q$ and so $r=q_1+r_1ab\in Q$.
Then $Q=(Q+Rab)\cap(Q+Rac)$, which contradicts the assumption that $Q$ is
irreducible.
\end{proof}

A ring $R$ is said to be a \textit{Boolean ring} if $x=x^2$ for all $x\in R$%
. It is famous that every prime ideal in a Boolean ring $R$ is maximal.
Notice that every ideal of a Boolean ring $R$ is radical. So, every
(uniformly) 2-absorbing primary ideal of $R$ is a 2-absorbing ideal of $R$.

\begin{corollary}
Let $R$ be a Boolean ring. Then every irreducible ideal of $R$ is a maximal
ideal.
\end{corollary}

\begin{proof}
Let $I$ be an irreducible ideal of $R$. Thus, Theorem \ref{main3} implies
that $I$ is special 2-absorbing primary. Therefore by Proposition \ref{rad},
either $I=\sqrt{I}$ is a maximal ideal or is the intersection of two
distinct maximal ideals. Since $I$ is irreducible, then $I$ cannot be in the
second form. Hence $I$ is a maximal ideal.
\end{proof}

\begin{proof}
By Proposition \ref{rad} and Theorem \ref{main3}.
\end{proof}

\begin{proposition}
Let $Q$ be a special 2-absorbing primary ideal of $R$ and $\mathfrak{p},~%
\mathfrak{q}$ be distinct prime ideals of $R$.

\begin{enumerate}
\item If $\sqrt{Q}=\mathfrak{p}$, then $\{(Q:_Rx)\mid x\in R\backslash%
\mathfrak{p}\}$ is a totally ordered set.

\item If $\sqrt{Q}=\mathfrak{p}\cap\mathfrak{q}$, then $\{(Q:_Rx)\mid x\in
R\backslash\mathfrak{p}\cup\mathfrak{q}\}$ is a totally ordered set.
\end{enumerate}
\end{proposition}

\begin{proof}
(1) Let $x,y\in R\backslash\mathfrak{p}$. Then $xy\in R\backslash\mathfrak{p}
$. It is clear that $(Q:_Rx)\cup(Q:_Ry)\subseteq(Q:_Rxy)$. Assume that $%
r\in(Q:_Rxy)$. Therefore $rxy\in Q$, whence $rx\in Q$ or $ry\in Q$, because $%
xy\notin\sqrt{Q}$. Consequently $(Q:_Rxy)=(Q:_Rx)\cup(Q:_Ry)$. Thus, either $%
(Q:_Rxy)=(Q:_Rx)$ or $(Q:_Rxy)=(Q:_Ry)$, and so either $(Q:_Ry)%
\subseteq(Q:_Rx)$ or $(Q:_Rx)\subseteq(Q:_Ry)$.

(2) Is similar to the proof of (1).
\end{proof}

\begin{corollary}
Let $f:R\longrightarrow R^{\prime }$ be a homomorphism of
commutative rings. Then the following statements hold:

\begin{enumerate}
\item If $Q^{\prime }$ is a special 2-absorbing primary ideal of $%
R^{\prime }$, then $f^{-1}(Q^{\prime })$ is a special 2-absorbing primary
ideal of $R$.

\item If $f$ is an epimorphism and $Q$ is a special 2-absorbing primary
ideal of $R$ containing $\ker (f)$, then $f(Q)$ is a special 2-absorbing
primary ideal of $R^{\prime }$.
\end{enumerate}
\end{corollary}
\begin{proof}
By Proposition \ref{monoepi}.
\end{proof}

Let $R$ be a ring with identity. We recall that if $f=a_0+a_1X+\cdots+a_tX^t$ 
is a polynomial on the ring $R$, then {\it content} of $f$ is defined as the ideal of $R$, generated by the coefficients of $f$, i.e.
$c( f )=(a_0,a_1,\dots,a_t)$. Let $T$ be an $R$-algebra and $c$ the function from $T$ to the ideals of $R$ defined by $c(f)=\cap\{I\mid$ $I$ \mbox{is an ideal of} $R$ \mbox{and} $f\in IT\}$ known as the content of $f$. Note that the content function $c$ is nothing but the generalization of the content of a polynomial $f\in R[X]$. The $R$-algebra $T$ is called a {\it content $R$-algebra} if the following conditions hold:
\begin{enumerate}
\item For all $f\in T$, $f\in c(f)T$.
\item (Faithful flatness ) For any $r\in R$ and $f\in T$, the equation $c(rf )=rc(f)$ holds and
$c(1_T)=R$.
\item (Dedekind-Mertens content formula) For each $f,g\in T$, there exists a natural
number $n$ such that $c(f)^nc(g)=c(f)^{n-1}c(fg)$.
\end{enumerate}
For more information on content algebras and their examples we refer to \cite{north}, \cite{ohm} and \cite{rush}.
In \cite{nas} Nasehpour gave the definition of a Gaussian $R$-algebra as follows: Let $T$ be an $R$-algebra such that $f\in c(f)T$ for all $f\in T$. $T$ is said to be a Gaussian $R$-algebra if $c(fg)=c( f )c(g)$, for all $f,g\in T$.
\begin{example}(\cite{nas})
Let $T$ be a content $R$-algebra such that $R$ is a Pr\"{u}fer domain. Since every
nonzero finitely generated ideal of $R$ is a cancellation ideal of $R$, the Dedekind-Mertens 
content formula causes $T$ to be a Gaussian $R$-algebra.
\end{example}

\begin{theorem}\label{pruf1}
Let R be a Pr\"{u}fer domain, $T$ a content $R$-algebra and $Q$ an ideal of $R$. 
Then $Q$ is a special 2-absorbing primary ideal of $R$ if and only if $QT$ is a special 2-absorbing primary ideal of $T$.
\end{theorem}
\begin{proof}
$(\Rightarrow)$ Assume that $Q$ is a special 2-absorbing primary ideal of $R$. Let $fgh\in QT$ for some $f,g,h\in T$. Then $c(fgh)\subseteq Q$. Since $R$ is a Pr\"{u}fer domain and $T$ is a content $R$-algebra, 
then $T$ is a Gaussian $R$-algebra. Therefore $c(fgh)=c(f)c(g)c(h)\subseteq Q$. Since $Q$ is a special 2-absorbing primary ideal of $R$,
Theorem \ref{special} implies that either
$c(f)c(g)=c(fg)\subseteq Q$ or $c(f)c(h)=c(fh)\subseteq Q$ or $c(g)c(h)=c(gh)\subseteq\sqrt{Q}$. So $fg\in c(fg)T\subseteq QT$ or $fh\in c(fh)T\subseteq QT$
or $gh\in c(gh)T\subseteq\sqrt{Q}T\subseteq\sqrt{QT}$. Consequently $QT$ is a special 2-absorbing primary ideal of $T$.\newline
$(\Leftarrow)$ Note that since $T$ is a content $R$-algebra, $QT\cap R=Q$ for every ideal $Q$ of $R$. Now, apply Corollary \ref{frac}(1).
\end{proof}

The algebra of all polynomials over an arbitrary ring with an arbitrary number of indeterminates is an example of content algebras.
\begin{corollary}\label{last}
Let $R$ be a Pr\"{u}fer domain and $Q$ be an ideal of $R$. Then $Q$ is a special 2-absorbing primary ideal of $R$
if and only if $Q[X]$ is a special 2-absorbing primary ideal of $R[X]$.
\end{corollary}

\begin{corollary}
Let $S$ be a multiplicatively closed subset of $R$ and $Q$ be a proper ideal
of $R$. Then the following conditions hold:
\begin{enumerate}
\item If $Q$ is a special 2-absorbing primary ideal of $R$ such that $%
Q\cap S=\emptyset $, then $S^{-1}Q$ is a special 2-absorbing primary ideal
of $S^{-1}R$ with $\mbox{2-}ord(S^{-1}Q)\leq \mbox{2-}ord(Q)$.

\item If $S^{-1}Q$ is a special 2-absorbing primary ideal of $S^{-1}R$ and 
$S\cap Z_{Q}(R)=\emptyset $, then $Q$ is a special 2-absorbing primary
ideal of $R$ with $\mbox{2-}ord(Q)\leq \mbox{2-}ord(S^{-1}Q).$
\end{enumerate}
\end{corollary}
\begin{proof}
By Proposition \ref{multi}.
\end{proof}

In view of Theorem \ref{product1} and its proof, we have the following
result.
\begin{corollary}\label{result1} 
Let $R=R_1\times R_2$, where $R_1$ and $R_2$
are rings with $1\neq0$. Let $Q$ be a proper ideal of $R$. Then the
following conditions are equivalent:

\begin{enumerate}
\item $Q$ is a special 2-absorbing primary ideal of $R$;

\item Either $Q=Q_1\times R_2$ for some special 2-absorbing primary ideal $%
Q_1$ of $R_1$ or $Q=R_1\times Q_2$ for some special 2-absorbing primary
ideal $Q_2$ of $R_2$ or $Q=Q_1\times Q_2$ for some prime ideal $Q_1$ of $R_1$
and some prime ideal $Q_2$ of $R_2$.
\end{enumerate}
\end{corollary}

\begin{corollary}\label{product4} 
Let $R=R_1\times R_2$, where $R_1$ and $R_2$ are rings with $%
1\neq0$. Suppose that $Q_1$ is a proper ideal of $R_1$ and $Q_2$ is a proper
ideal of $R_2$. Then $Q_1\times Q_2$ is a special 2-absorbing primary ideal
of $R$ if and only if it is a 2-absorbing ideal of $R$.
\end{corollary}

\begin{proof}
See Corollary \ref{result1} and apply \cite[Theorem 4.7]{AB1}.
\end{proof}

\begin{corollary}
Let $R=R_1\times R_2\times\cdots\times R_n$, where $2\leq n<\infty$, and $%
R_1,R_2, ...,R_n$ are rings with $1\neq0$. For a proper ideal $Q$ of $R$ the
following conditions are equivalent:

\begin{enumerate}
\item $Q$ is a special 2-absorbing primary ideal of $R$.

\item Either $Q=\times^n_{t=1}Q_t$ such that for some $k\in\{1,2,...,n\}$, $%
Q_k$ is a special 2-absorbing primary ideal of $R_k$, and $Q_t=R_t$ for
every $t\in\{1,2,...,n\}\backslash\{k\}$ or $Q=\times_{t=1}^nQ_t$ such that
for some $k,m\in\{1,2,...,n\}$, $Q_k$ is a prime ideal of $R_k$, $Q_m$ is a
prime ideal of $R_m$, and $Q_t=R_t$ for every $t\in\{1,2,...,n\}\backslash%
\{k,m\}$.
\end{enumerate}
\end{corollary}
\begin{proof}
By Theorem \ref{product3}.
\end{proof}

\vspace{5mm} \noindent \footnotesize 
\begin{minipage}[b]{10cm}
Hojjat Mostafanasab \\
Department of Mathematics and Applications, \\ 
University of Mohaghegh Ardabili, \\ 
P. O. Box 179, Ardabil, Iran. \\
Email: h.mostafanasab@uma.ac.ir, \hspace{1mm} h.mostafanasab@gmail.com
\end{minipage}\\

\vspace{2mm} \noindent \footnotesize
\begin{minipage}[b]{10cm}
\"{U}nsal Tekir and G\"{u}l\c{s}en Ulucak\\
Department of Mathematics, \\ 
Marmara University, \\ 
Ziverbey, Goztepe, Istanbul 34722, Turkey. \\
Email: utekir@marmara.edu.tr,\hspace{1mm} gulsenulucak58@gmail.com
\end{minipage}


\begin{thebibliography}{99}
\bibitem{AB1} D. F. Anderson and A. Badawi, On $n$-absorbing ideals of
commutative rings, \textit{Comm. Algebra} \textbf{39} (2011) 1646--1672.

\bibitem{AKL} D. D. Anderson, K. R. Knopp and R. L. Lewin, Ideals generated
by powers of elements, \textit{Bull. Austral. Math. Soc.,} \textbf{49} (1994)
373--376.


\bibitem{B} A. Badawi, On $2$-absorbing ideals of commutative rings, Bull.
Austral. Math. Soc., \textbf{75} (2007), 417--429.

\bibitem{BTY} A. Badawi, \"{U}. Tekir and E. Yetkin, On $2$-absorbing primary
ideals in commutative rings, Bull. Korean Math. Soc., \textbf{51} (4)
(2014), 1163--1173.

\bibitem{YB} A. Badawi and A. Yousefian Darani, On weakly $2$-absorbing
ideals of commutative rings, \textit{Houston J. Math.}, \textbf{39} (2013),
441--452.

\bibitem{CH} J. A. Cox and A. J. Hetzel, Uniformly primary ideals, {\it J. Pure
Appl. Algebra,} \textbf{212} (2008), 1-8.


\bibitem{Hoc} M. Hochster, Criteria for equality of ordinary and symbolic powers of primes, {\it Math. Z.} \textbf{133} (1973) 53-65.

\bibitem{Huk} J. Hukaba, {\it Commutative rings with zero divisors}, 
Marcel Dekker, Inc., New York, 1988.


\bibitem{Mos} H. Mostafanasab, E. Yetkin, U. Tekir and A. Yousefian Darani,
On $2$-absorbing primary submodules of modules over commutative rings, 
\textit{An. \c{S}t. Univ. Ovidius Constanta,} (in press)

\bibitem{nas} P. Nasehpour, {On the Anderson-Badawi $\omega_{R[X]}(I[X])=\omega_R(I)$ conjecture}, arXiv:1401.0459, (2014).

\bibitem{north} D. G. Northcott, {A generalization of a theorem on the content of polynomials}, {\it Proc. Cambridge Phil.
Soc.,} {\bf55} (1959), 282--288.

\bibitem{ohm} J. Ohm and D. E. Rush, { Content modules and algebras}, {\it Math. Scand.}, {\bf31} (1972), 49--68.

\bibitem{rush} D. E. Rush, { Content algebras}, {\it Canad. Math. Bull.}, {\bf21} (3) (1978), 329--334.

\bibitem{Sh} R.Y. Sharp, {\it Steps in commutative algebra},  Second edition, Cambridge University Press, Cambridge, 2000.

\bibitem{YF} A. Yousefian Darani and F. Soheilnia, $2$-absorbing and weakly $%
2$-absorbing submoduels, \textit{Thai J. Math.} \textbf{9}(3) (2011)
577--584.

\bibitem{YF2} A. Yousefian Darani and F. Soheilnia, On $n$-absorbing
submodules, \textit{Math. Comm.}, \textbf{17} (2012), 547-557.
\end{thebibliography}
\end{document}